\documentclass[11pt, reqno]{amsart}
\usepackage{graphics,graphicx}
\usepackage{xcolor}
\usepackage{amsmath}
\usepackage{amsthm,amstext,amsfonts,bm,amssymb}
\usepackage{a4wide}
\usepackage{float}
\usepackage{enumerate}
\usepackage{lscape}

\newtheorem{theorem}{Theorem}[section]

\newtheorem{proposition}[theorem]{Proposition}

\theoremstyle{definition}
\newtheorem{definition}[theorem]{Definition}
\newtheorem{example}[theorem]{Example}

\theoremstyle{remark}
\newtheorem{remark}[theorem]{Remark}

\def\stirfir#1#2{\genfrac{[}{]}{0pt}{}{#1}{#2}}
\def\stirsec#1#2{\genfrac{\{}{\}}{0pt}{}{#1}{#2}}
\def\lah#1#2{\genfrac{\lfloor}{\rfloor}{0pt}{}{#1}{#2}}

\newcommand{\La}{{\mathbb L}}
\newcommand{\Pa}{{\mathcal P}}
\newcommand{\D}{{\mathcal D}}
\newcommand{\Z}{{\mathbb Z}}
\newcommand{\F}{{\mathbb F}}
\newcommand{\I}{{\mathbb I}}

\begin{document}

\title{Generalized Ordered Set Partitions}


\author{Be\'ata B\'enyi}
\address{\noindent Faculty of Water Sciences, National University of Public Service, Budapest, HUNGARY}
\email{beata.benyi@gmail.com}

\author{Miguel M\'{e}ndez}
\address{\noindent School of Mathematical and Computational Sciences, Yachay University, Urcuqui, ECUADOR}
\email{mmendez@yachaytech.edu.ec}

\author{Jos\'{e} L. Ramirez}
\address{Departamento de  Matem\'{a}ticas,
Universidad Nacional de Colombia, Bogot\'{a}, COLOMBIA}
\email{jlramirezr@unal.edu.co}
%


\date{\today}


\begin{abstract}
In this paper, we consider ordered set partitions obtained by imposing conditions on the size of the lists, and such that the first $r$ elements are in distinct blocks, respectively. We introduce a generalization of the Lah numbers. For this new combinatorial sequence we derive its exponential generating function, some recurrence relations,  and combinatorial identities. We prove and present results using combinatorial arguments, generating functions, the symbolic method and Riordan arrays. For some specific cases we provide a combinatorial interpretation for the inverse matrix of the generalized Lah numbers by means of two families of posets. 
\end{abstract}

\maketitle

\newcommand{\N}{{\mathbb N}}
\newcommand{\realpart}{\mathop{\rm Re}\nolimits}
\newcommand{\imagpart}{\mathop{\rm Im}\nolimits}

\numberwithin{equation}{section}

\section{Introduction} \label{intro}
\setcounter{equation}{0}

The (unsigned) Lah numbers, denoted by $\lah{n}{k}$, enumerate the number of partitions of a set with $n$ elements into $k$ non-empty  ordered lists.  This sequence satisfies the following recurrence
$$\lah{n}{k}=\lah{n-1}{k-1}+(n+k-1)\lah{n-1}{k},$$
with the initial values $\lah{0}{0}=1$ and $\lah{n}{0}=\lah{0}{n}=0$ if $n\geq 1$.

They can be computed by the following explicit  formula (cf.\ \cite{Riordan})
\begin{align*}
\lah{n}{k}=\frac{n!}{k!}\binom{n-1}{k-1}.
\end{align*}
From the above equation it is possible to obtain the exponential generating function
$$\sum_{n\geq k}\lah{n}{k}\frac{x^n}{n!}=\frac{1}{k!}\left(\frac{x}{1-x}\right)^k.$$
The Lah numbers can also be defined  as the connecting coefficients between the rising and  falling factorial polynomials
\begin{align}\label{id1}
(x)^n=\sum_{k=0}^n\lah{n}{k}(x)_k, \quad  (n\geq 0),
\end{align}
 where $(x)^n=x(x+1)\cdots (x+n-1)$ and $(x)_n=x(x-1)\cdots (x-n+1)$, for $n\geq  1$, with the initial values  $(x)^0=1=(x)_0$.

The Lah numbers are related to Stirling numbers by the following orthogonality relation
\begin{equation}\label{orto}
\lah{n}{k}=\sum_{j=k}^n\stirfir{n}{j}\stirsec{j}{k}, \qquad 0 \leq k \leq n,
\end{equation}
where $\stirfir{n}{m}$ and $\stirsec{n}{m}$ are the Stirling numbers of the first and second kind, respectively.

Let us introduce the sequence $L(n)$ as the total number of partitions of $[n]$
into ordered lists  (also called fragmented permutations \cite{Flajolet}) so that
 $$L(n):=\sum_{k=0}^n\lah{n}{k}.$$
The first few terms are
$$1,\, 1,\, 3,\, 13,\, 73,\, 501,\, 4051,\, 37633,\, 394353, \, 4596553, \, 5894109, \dots $$
For example, $L(3) = 13$, the ordered lists being
  \begin{align*}
 & \left\{ \{1 \}, \{ 2 \}, \{ 3 \} \right\}, & & \left\{ \{ 1, \, 2 \}, \{ 3 \}  \right\}, & & \left\{ \{ 2, \, 1 \}, \{ 3 \}  \right\}, &  &  \left\{ \{ 1, \, 3 \}, \{ 2 \}  \right\}, & & \left\{ \{ 3, \, 1 \}, \{ 2 \}  \right\}, & \\
  & \left\{  \{ 1 \},  \{ 2, \, 3 \}\right\}, & &   \left\{  \{ 1 \},  \{ 3, \, 2 \}\right\}, &  & \left\{ \{ 1, \, 2, \, 3 \} \right\}, & & \left\{ \{ 1, \, 3, \, 2 \} \right\}, & &  \left\{ \{ 2, \, 1, \, 3 \} \right\}, & \\
  &  \left\{ \{ 2, \, 3, \, 1 \} \right\},  & &  \left\{ \{ 3, \, 1, \, 2 \} \right\},  &&  \left\{ \{ 3, \, 2, \, 1 \} \right\}. &
  \end{align*}
The exponential generating function is given by
$$\sum_{n=0}^\infty L(n)\frac{x^n}{n!}=e^{x/(1-x)},$$
and satisfies the recurrence relation (cf.\ \cite{MS})
\begin{align*}
L(n+1)=(2n+1)L(n)-(n^2-n)L(n-1), \quad n\geq 1 
\end{align*}
with the initial values $L(0)=L(1)=1$.
The asymptotic behavior of the sequence can be derived using the saddle point method (See \cite{Flajolet}, VIII. 7., p.562.):
$$
\frac{L(n)}{n!}\sim \frac{e^{-\frac{1}{2}}e^{2\sqrt{n}}}{2\sqrt{\pi}n^{\frac{3}{4}}}.
$$

In this paper, we study the number of partitions of $[n]$ into $k$ non-empty lists (ordered blocks) such that the size $s$ of each list belongs to a given set $S$. This sequence is called \emph{$S$-Lah numbers} (cf.\ \cite{Engbers-2016, mihoubi-2017a}). We use the Karamata-Knuth notation $\lah{n}{k}_S$ for this sequence. Notice that if $S=\Z^+=\{1,2,3,\ldots\}$ we recover the classical Lah numbers. If we take $S=\{1, 2,  \dots, m\}$, we obtain the \emph{restricted Lah numbers} \cite{ManSha}. In a similar way, if we take $S=\{m, m+1, \dots\}$, we have the \emph{associated Lah numbers}  introduced by Belbachir and Bousbaa \cite{Belbachir-2016}. Inspired by the well-known $r$-Stirling numbers introduced by Broder \cite{Broder}, we introduce $(S,r)$-Lah numbers, denoted by $\lah{n}{k}_{S,r}$, as the number of partitions of $[n + r]$ into $k + r$ non-empty ordered lists with the additional condition that the first $r$ elements are in distinct blocks. It is clear that $\lah{n}{k}_{S,0}=\lah{n}{k}_{S}$.\\

The outline of the paper is as follows. First, we investigate $S$-Lah numbers, derive the generating function by the symbolic method and further identities using combinatorial arguments. In the next sections we study the $(S,r)$-Lah numbers. In Section 3 we use classical combinatorial arguments and the symbolic method, in Section 4 we use the theory of Riordan arrays for the study of the $(S,r)$-Lah matrix and its inverse. In Section 5 we provide a new combinatorial interpretation of the $(S,r)$-Lah numbers (that involves also the $S$-Lah numbers), and define a partial order on the underlying set such that the M\"obius cardinal is given by the entries of the inverse of the $(S,r)$-Lah matrix. In Section 6 we complete our study with the introduction and characterization of $(S,r)$-Fubini numbers, which count the lists of  blocks with the extra conditions on the sizes and on the elements $1,\ldots, r$ (using matrix theory, combinatorial arguments and the symbolic method). Finally, we show some results on the number of doubly ordered partitions, ``lists of lists'', with our conditions on the size of the lists and  on the elements $[r]$.

\section{$S$-Lah numbers}
Restrictions and generalizations of Stirling numbers of the second and first kind were studied recently by many authors, but these versions of Lah numbers have received less attention yet. Engbers  et al. \cite{Engbers-2016} introduced the \emph{$S$-Lah numbers} as the number of partitions of $[n]$ into $k$ non-empty lists (ordered blocks) such that the size $s$ of each list belongs to a given set $S$. For further applications of this sequence see \cite{BenyiRamirez}.

For the sake of completeness, we mention here another combinatorial interpretation of the $S$-Lah numbers in terms of Dyck paths. In particular, Callan \cite{Callan} showed that the Lah number $\lah{n}{k}$ counts $n$-Dyck paths with $n+1-k$ labelled peaks. According to Callan's bijection, the $S$-restricted Lah numbers counts the $n$-Dyck paths with $n+1-k$ labelled peaks, such that the length of any sequence of  consecutive peaks is contained in $S$. A \emph{peak} of a Dyck path is an up step followed directly by a down step: $(UD)$, and a sequence of consecutive peaks is a sequence of up-down step pairs $UDUDUDUD=(UD)^4$, while the length of such a $(UD)^m$ is $m$.

The exponential generating function can immediately be obtained using the symbolic method \cite{Flajolet}. Let $S$ be a given set of integers. Then, the construction of a partition of lists of sizes containing in $S$ is
\[\mbox{SET}_k(\mbox{SEQ}_S(\mathcal{X})),\] where $\mbox{SET}_k(\mathcal{X})$ denotes $k$-sets of objects and $\mbox{SEQ}_S(\mathcal{X})$ denotes lists (sequences) of sizes belonging to the set $S$. The construction above directly translates to
$$
\sum_{n=k}^{\infty}{\lah{n}{k}_S}\frac{x^n}{n!} = \frac{1}{k!}\left({\sum_{s\in S}x^s}\right)^k.
$$

In particular, for $S$ being the odd ($\mathcal{O}$), respectively the even numbers $\mathcal{E}$, we have:
\begin{align*}
\sum_{n=k}^{\infty}{\lah{n}{k}_{\mathcal{O}}}\frac{x^n}{n!} &= \frac{1}{k!}\left(\frac{x}{1-x^2}\right)^k,\\
\sum_{n=k}^{\infty}{\lah{n}{k}_{\mathcal{E}}}\frac{x^n}{n!} &= \frac{1}{k!}\left(\frac{x^2}{1-x^2}\right)^k.
\end{align*}
The generating functions for the associated  \cite{Belbachir-2016}  and restricted Lah numbers are also immediate:
\begin{align*}
\sum_{n=mk}^{\infty}{\lah{n}{k}_{\geq{m}}}\frac{x^n}{n!} &= \frac{1}{k!}\left(\frac{x^m}{1-x}\right)^k,\\
\sum_{n=k}^{\infty}{\lah{n}{k}_{\leq{m}}}\frac{x^n}{n!} &= \frac{1}{k!}\left(\frac{x-x^{m+1}}{1-x}\right)^k.
\end{align*}
Similarly, if we do not care about the number of lists, the construction is modified as
\[\mbox{SET}(\mbox{SEQ}_S(\mathcal{X})).\]

Hence, for the $S$-restricted version of the sequence $L(n)$ we have:
$$
\sum_{n=0}^{\infty}{L_S(n)}\frac{x^n}{n!} = \exp\left({\sum_{s\in S}x^s}\right),
$$
where  $L_S(n)=\sum_{k=0}^n\lah{n}{k}_{S}$.

Note that the $S$-restricted Lah sequence is a  particular case of the partial Bell polynomials \cite{comtet-1974a}. Therefore, if $S=\{k_1, k_2, \dots \}$, then they are given by
$$\lah{n}{k}_S=\sum_{\substack{c_1k_1+c_2k_2+\cdots =n, \\ c_1+c_2+\cdots = k}}\frac{n!}{c_1!c_2!\cdots}.$$

Moreover, it is possible to generalize the Identity \eqref{id1} by means of the potential polynomials (\cite[Theorem B, pp. 141]{comtet-1974a}). Let $f_{S,t}(x)$ be the function defined by
$$f_{S,t}(x):=(1+H_S(x))^t,$$
where $H_S(x)=\sum_{s\in S} x^{s}.$
Then
 \begin{align*}
\left. \frac{d^n}{dx^n}f_{S,t}(x)\right |_{x=0}:=f_{S,t}^{(n)}(0)=\sum_{k=0}^n\lah{n}{k}_{S} (t)_k.
\end{align*}

For a given set $S$ of integers, let $C_S(n,k)$ denote the number of compositions of $n$ into exactly $k$ parts such that the size of each part is included in $S$.  Then we have the following relation
$$
\lah{n}{k}_S=\frac{n!}{k!}C_S(n,k).
$$
Clearly,
$$\lah{n}{k}_S=\sum_{i_1+i_2+\cdots +i_k=n\atop i_j\in S}\frac{(n)_{i_1}(n-i_1)_{i_2}(n-(i_1+i_2))_{i_3}\cdots (n-(i_1+\cdots+i_{k-1}))_{i_k}}{k!},$$
which reduces to the formula above.

For $S=\{m,m+1, \ldots \}$, we recover the formula of the associated Lah numbers given in \cite{Belbachir-2016} with $C_S(n,k)=\binom{n-(m-1)k-1}{k-1}$. In particular, the number of partitions $[n]$ into $k$ lists without singletons
is $\lah{n}{k}_{\geq 2}=\frac{n!}{k!}\binom{n-k-1}{k-1}$. Further, using the results on compositions of Heubach and Mansour \cite{Heubach-2004},
we have for the sets of odd, respectively for even integers the following formulas.
\begin{align*}
\lah{2n-k}{k}_{\mathcal{O}}=\frac{(2n-k)!}{k!}\binom{n-1}{k-1} \quad \text{and} 	\quad \lah{2n}{k}_{\mathcal{E}}=\frac{2n!}{k!}\binom{n-1}{k-1}.
\end{align*}
Hence, we have
$$
(2n)_{k}\lah{2n-k}{k}_{\mathcal{O}}=\lah{2n}{k}_{\mathcal{E}}.
$$

Next, we derive some recurrences for the $S$-Lah numbers.

\begin{theorem}
For integers $n,k$ and a given set of integers $S$, we have
$$
\lah{n}{k}_S=\sum_{s\in S}s(n-1)_{s-1}\lah{n-s}{k-1}_S.
$$
\end{theorem}

\begin{proof}
Consider an $n$ element set partitioned into $k$ non-empty ordered blocks. Assume that the $n$th element is contained in a set of size $s\in S$. For this block we choose in $\binom{n-1}{s-1}$ ways the other elements, and order the elements in the block in $s!$ ways. The remaining elements can be partitioned in $\lah{n-s}{k-1}_S$ ways.
\end{proof}

We can derive a recursive formula for $\lah{n}{k}_S$, where we do not need to sum over the whole set $S$, which means a simplification for some set $S$. A set $S$ of integers is the disjoint union of sequences of consecutive integers: $S_i=\{i, i+1,\ldots, i+s_i-1\}$, for some $i$. The least and greatest elements of the sets $S_i$ play an important role; hence, we define $S^*$ to be the set of least elements, and $\overline{S}$ the set of greatest elements of the sets $S_i$. Further, we let $\widehat{S}$ denote the set of greatest elements that are not least elements: $\widehat{S}=S^*-\overline{S}$. Alternative definitions of $S^*$ and $\overline{S}$ are the following: $S^*:=\{s\in S|s-1\not\in S\}$ and $\overline{S}:=\{s\in S|s+1\not\in S\}$, respectively.
\begin{theorem} Given a set $S$, let $S^*$ and $\widehat{S}$ be sets as defined above. We have
$$
\lah{n}{k}_S=(n+k-1)\lah{n-1}{k}_{S}+\sum_{s\in S^*}\binom{n-1}{s-1}s!\lah{n-s}{k-1}_S -\sum_{s\in \widehat{S}}\binom{n-1}{s}(s+1)!\lah{n-s-1}{k-1}_S.
$$
\end{theorem}
\begin{proof}
The left hand side counts the partitions of $n$ elements into lists such that each list has size included in $S$.
Consider the partition into lists of $n-1$ elements. We insert the $n$th element before each element or as a last element of any list. (This can be done in $(n+k-1)\lah{n-1}{k}_S$ ways.) But we do not obtain all the partitions this way, since the partitions in which the $n$th element is in a list of size $s$ with $s\in S^*$ are missing. $\sum_{s\in S^*}\binom{n-1}{s-1}s!\lah{n-s}{k-1}_S$ counts the number of such partitions. Moreover, we obtained by the insertion partitions of $n$ for that not every list has the required size. This happens, if the $n$th element is inserted into a list of size $s$, where $s\in \widehat{S}$. Since the number of such partitions is $\sum_{s\in \widehat{S}}\binom{n-1}{s}(s+1)!\lah{n-s-1}{k-1}_S$, we need to reduce our sum by this number.
\end{proof}
For $S=\Z^+$, the formula is the well-known recurrence of the Lah numbers. For $S=\{s, s+1, s+2, \ldots\}$, we recover the recurrence for the associated Lah numbers \cite{Belbachir-2016}:
$$
\lah{n}{k}_{\geq s}=(n+k-1)\lah{n-1}{k}_{\geq s}+\binom{n-1}{s-1}s!\lah{n-s}{k-1}_{\geq s}.
$$
Setting $S=\{1,2,\ldots, s\}$, we obtain the recurrence relation for  the restricted Lah numbers
$$\lah{n}{k}_{\leq s}=(n+k-1)\lah{n-1}{k}_{\leq s}+\lah{n-1}{k-1}_{\leq s}-\binom{n-1}{s}(s+1)!\lah{n-s-1}{k-1}_{\leq s}.
$$
But it is also immediate to give the recurrence for the number of partitions into lists of $n$ element that do not contain lists of a given size, say $p(\not=1)$. This means namely that $S=S_{\overline{p}}=\Z^+\setminus\{p\}$; hence, $S_{\overline{p}}^*=\{p+1\}$ and $\widehat{S_{\overline{p}}}=\{p-1\}$. We have
$$
\lah{n}{k}_{S_{\overline{p}}}=(n+k-1)\lah{n-1}{k}_{S_{\overline{p}}}+\binom{n-1}{p}(p+1)!\lah{n-p-1}{k-1}_{S_{\overline{p}}}-
\binom{n-1}{p-1}p!\lah{n-p+1}{k-1}_{S_{\overline{p}}}.
$$

\section{The $r$-Version}\label{rversion1}
Now, we turn our attention to the $S$-restricted case of $r$-Lah numbers.  
Given a set of positive integers $S$, let $\lah{n}{k}_{S,r}$ denote the number of partitions of $n+r$ elements into $k+r$ lists such that the size of each list is contained in $S$ and the first $r$ elements are contained in distinct lists. We call the first $r$ elements \emph{special elements} and a list that contains a special element \emph{special list}. This generalization is inspired by the well-known $r$-Stirling numbers introduced by Broder \cite{Broder}.  For $S=\Z^*$ we obtain the $r$-Lah numbers studied recently by several authors, \cite{Nyul-2015, Belbachir-2013, Belbachir-2014}, and for $S=\{1,2,\ldots, n\}$ we obtain the restricted $r$-Lah numbers that were introduced by Shattuck in \cite{Shattuck}.

Theorem \ref{teo41} provides the relation between the $S$-Lah numbers and $(S,r)$-Lah numbers.
\begin{theorem}\label{teo41}
Let $n,k\geq 1$ and $n>r$ be integers, and $S$ a set of integers. We have the combinatorial identity
\[\lah{n}{k}_{S,r}= \sum_{m=0}^{n-k}\binom{n}{m}\sum_{i_1+\cdots+i_r=m\atop i_j+1\in S}m! (i_1+1)\cdots (i_r+1)\lah{n-m}{k}_S.\]
\end{theorem}
\begin{proof}
Let $i_j$ be the number of non-special elements that are contained in the list of the special element $j$. Further, let $m$ be the sum of $i_j$, i.e., the number of non-special elements that are contained in any of the special lists. Fix $m$, and construct the special lists and the non-special lists separately.
$$
\lah{n}{k}_{S,r}=\sum_{m=0}^{n-k}\sum_{i_1+\cdots+i_r=m\atop i_j+1\in S} (i_1+1)!\cdots(i_r+1)!\binom{n}{i_1,i_2,\ldots,i_r,n-m}\lah{n-m}{k}_S.
$$
After simplification we get the above formula.
\end{proof}

\begin{theorem}
Let $n,k\geq 1$ and $n>r$ be integers, and $S$ a set of integers. The $(S,r)$-Lah numbers satisfy the following recurrence relation
$$
\lah{n+1}{k}_{S,r}=\lah{n}{k-1}_{S,r+1}+r\sum_{s\in S}s!\binom{n}{s-2}\lah{n-s+2}{k}_{S,r-1}.
$$
\end{theorem}
\begin{proof}
Assume that the $(n+1)$-th (non-special) element is in a list that does not contain any of the $r$ special elements. Then, we can consider $(n+1)$ as an extra special element; hence, the number of such partitions is $\lah{n}{k-1}_{S,r+1}$ by definition. Assume now that the $(n+1)$-th element is contained in one of the $r$ special lists. Then, first we choose the special list in $r$ ways, then we choose $s-2$ elements out of $n$.  We permute now these $s$ elements; $s-2$ non-special elements, the special element, and ($n+1$), in order to obtain the list that contains the element $(n+1)$. The remaining $n-(s-2)+(r-1)$ elements construct a $(S,r-1)$-partition into $k$ non-empty lists, which is counted by definition by $\lah{n-s+2}{k}_{S,r-1}$.
\end{proof}

In Theorem \ref{riden} we give several combinatorial identities for the $(S,r)$-Lah numbers.

\begin{theorem}\label{riden}
Let $n,k\geq 1$ and $n>r$ be integers, and $S$ a set of integers. The $(S,r)$-Lah numbers satisfy the following identities
\begin{align}
k\lah{n}{k}_{S,r}&=\sum_{s\in S}s!\binom{n}{s}\lah{n-s}{k-1}_{S,r};\label{rlahid1}\\
r\lah{n}{k}_{S,r}&=r\sum_{s\in S}s!\binom{n}{s-1}\lah{n-s+1}{k}_{S,r-1};\label{rlahid2}\\
(n+r)\lah{n}{k}_{S,r}&=\sum_{s\in S}s!s\binom{n}{s}\lah{n-s}{k-1}_{S,r}+r\sum_{s\in S}s!s\binom{n}{s-1}\lah{n-s+1}{k}_{S,r-1}.\label{rlahid3}
\end{align}
\end{theorem}
\begin{proof}
The left hand side of \eqref{rlahid1} counts  $(S,r)$ partitions into $k+r$ non-empty lists with one non-special list coloured. Count these partitions another way: first choose $s$ elements, that will be coloured, and construct a list in $s!\binom{n}{s}$ ways. The remaining $n-s+r$ elements form a partition into $k-1+r$ non-empty lists, such that the $r$ special elements are in distinct lists and the size of the lists are included in $S$.

The identities \eqref{rlahid2} and \eqref{rlahid3} follow similarly, colouring a special list, respectively an element.
\end{proof}

\begin{theorem}
Let $n,k\geq 1$ and $n>r$ be integers, and $S$ a set of integers. Further, let $u$ be an integer in $S$, $u\in S$. We have then
\begin{align}\label{rlahid4}
\lah{n}{k}_{S,r}=\sum_{i=0}^r\sum_{j=0}^{k}\binom{r}{i}(n)_{n-j}\frac{u^i}{(n-(u-1)i-uj)!}\lah{n-(u-1)i-uj}{k-j}_{S-\{u\},r-i}.
\end{align}
\end{theorem}
\begin{proof}
Let $i$ be the number of special lists of size $u$ and $j$ the number of non-special lists of size $u$. Choose the $i$ special elements in $\binom{r}{i}$ ways that are contained in a list of size $u$. Choose now for each of these $i$ special lists further $(u-1)$ elements, for the $j$ non-special lists $u$ elements and order the lists. These can be done in
$$
\frac{n!}{(u-1)!^iu!^jj!(n-(u-1)i-uj)!}(u!)^{i+j}
$$
ways. After simplification we obtain the formula.
\end{proof}
Setting $u=1$ into \eqref{rlahid4}, we obtain
$$
\lah{n}{k}_{S,r}=\sum_{i=0}^r\sum_{j=0}^k\binom{r}{i}\binom{n}{j}\lah{n-j}{k-j}_{S-\{1\}, r-i}.
$$

Finally, the Identity \eqref{id1} can be also generalized for the $(S,r)$-Lah numbers by using  Theorem 8 of \cite{mihoubi-2017a}. Let $f_{S,r,t}(x)$ be the function defined by
$$f_{S,r,t}(x):=(1+H_{S}(x))^t\left(\sum_{s\in S} sx^{s-1} \right)^r,$$
where $H_{S}(x)=\sum_{s\in S} x^{s}$.
Then
 \begin{align*}
\left. \frac{d^n}{dx^n}f_{S,r,t}(x)\right |_{x=0}:=f_{S,r,t}^{(n)}(0)=\sum_{k=0}^n\lah{n}{k}_{S,r} (t)_k.
\end{align*}

From the symbolic method it is possible to obtain the exponential generating function for the $(S,r)$-Lah numbers for a given set $S$. We have the construction
$$\mbox{SET}_k(\mbox{SEQ}_S(\mathcal{X})) \times \mbox{SEQ}_r(\Theta^*(\mbox{SEQ}_{S-1}(\mathcal{X}))).$$
In this construction $\Theta^*$ denotes a modification of the classical pointing operator. This operator is defined by a class $\mathcal B$ by
$$\mathcal{A}=\Theta^*\mathcal{B} \quad \text{iff} \quad \mathcal{A}_n=\{1,2,\ldots,n+1\}\times \mathcal{B}_n.$$
That is, in order to generate an  element in $\mathcal{A}$, create a gap for inserting our distinguished element,
in a list we can point to any element (and the gap is created after this pointed element) or the gap is at the beginning of the list, we insert the element as the starting element. So $A_n=(n+1)B_n$, then $A(x)=\frac{d}{dx}(xB(x))$.\\
From the construction above, we directly obtain the translation
\begin{align}\label{Sgenfunc}
\sum_{n=k}^{\infty}{\lah{n}{k}_{S,r}}\frac{x^n}{n!} = \frac{1}{k!}\left({\sum_{s\in S}x^s}\right)^k\left({\sum_{s\in S}sx^{s-1}}\right)^r.
\end{align}
Notice that if $L_{S,r}(n)$ denotes the total number of ordered $(S,r)$-partitions of an $n$-element set, then
$$
\sum_{n=0}^{\infty}L_{S,r}(n)\frac{x^n}{n!} = \exp\left({\sum_{s\in S}x^s}\right)\left({\sum_{s\in S}sx^{s-1}}\right)^r.
$$
\section{The $(S,r)$-Lah Matrix}

In this section we study the $(S,r)$-Lah matrix by using the theory of Riordan arrays \cite{Riordan2}. This theory is especially useful for the study of combinatorial matrices like Pascal matrix, Catalan matrix, Stirling matrices of both kinds, among other.

The \emph{$(S,r)$-Lah matrix} is the infinite matrix defined by
$$
\La_{S,r}:=\left[\lah{n}{k}_{S,r}\right]_{n, k \geq 0} .
$$
An infinite lower triangular matrix $L=\left[d_{n,k}\right]_{n,k\in \N}$ is called an \emph{exponential  Riordan array}, (cf.\ \cite{Barry}), if its column $k$ has generating function $g(x)\left(f(x)\right)^k/k!, k = 0, 1, 2, \dots$, where $g(x)$ and $f(x)$ are formal power series with $g(0) \neq 0$, $f(0)=0$ and $f'(0)\neq 0$.   The matrix corresponding to the pair $f(x), g(x)$ is denoted by  $\langle g(x),f(x)\rangle$.

If we multiply  $\langle g(x),f(x)\rangle$ by a column vector $(c_0, c_1, \dots)^T$ with exponential generating function $h(x)$, then the resulting column vector has exponential generating function $g(x)h(f(x))$. This property is known as the fundamental theorem of exponential Riordan arrays.  The product of two exponential Riordan arrays $\langle g(x),f(x)\rangle$ and $\langle h(x),\ell(x)\rangle$ is then defined by:
 $$\langle g(x),f(x)\rangle *\langle h(x),\ell(x)\rangle =\left\langle g(x)h\left(f(x)\right), \ell\left(f(x)\right)\right\rangle.$$
 The set of all exponential Riordan matrices is a group  under the operator $*$ (cf.\ \cite{Barry, Riordan2}).\\

 For example,  the Pascal matrix $\Pa$, the Stirling matrix of the second kind  $\mathcal{S}_2$,  and the Stirling matrix of the first kind  $\mathcal{S}_1$  are all given by the Riordan matrices:

 \begin{align*}
 \Pa=\langle e^x,x\rangle=&\left[\binom nk \right]_{n, k \geq 0}, \quad \quad   \mathcal{S}_2=\langle 1,e^x-1\rangle=\left[{n \brace k} \right]_{n, k \geq 0},  \quad \quad   \\
&\mathcal{S}_1=\langle 1, -\ln(1-x) \rangle=\left[{n \brack k} \right]_{n, k \geq 0}.
 \end{align*}

From Equation \eqref{Sgenfunc}, and the definition of Riordan matrix we obtain the following theorem.
\begin{theorem}
For all $S\subseteq \Z^+$ with $1\in S$, the matrix $\La_{S,r}$ is an exponential Riordan matrix given by
$$\La_S=\left\langle \left(\sum_{s\in S} sx^{s-1}\right)^r, \sum_{s \in S} x^{s}  \right\rangle.$$
\end{theorem}
Note that the row sum of the matrix  $\La_{S,r}$ is the sequence $L_{S,r}(n)$.

The inverse exponential Riordan array of $\La_{S,r}$ is denoted by
$$\F_{S,r}:=\left[\lah{n}{k}_{S,r}^{-1}\right]_{n, k \geq 0}.$$

In the following section we give a combinatorial interpretation for the absolute  values of the entries $\lah{n}{k}_{S,r}^{-1}$.   Note that Engbers et al. \cite{Engbers-2016}  give an  interesting combinatorial interpretation for the case $r=0$.

Since $\La_{S,r}*\F_{S,r} = \I,$ where $\I$ is the identity matrix, we have the orthogonality relation:
\begin{align*}
\sum_{i=k}^n \lah{n}{i}_{S,r}  \lah{i}{k}_{S,r}^{-1}&=\sum_{i=k}^n  \lah{n}{i}_{S,r}^{-1} \lah{i}{k}
=\delta_{k,n}.
\end{align*}
For the case $S=\Z^+$ and $r=0$ we recover the Equation \eqref{orto}.
From the orthogonality relation, we obtain the inverse relation:
\begin{align*}
f_n=\sum_{k=0}^n \lah{n}{k}_{S,r}^{-1} g_k  \iff g_n&=\sum_{k=0}^n \lah{n}{k}_{S,r}f_k.
\end{align*}

Now we define the \emph{$(S,r)$-Lah polynomials} by the combinatorial expression
$$L_{n, S, r}(x):=\sum_{k=0}^n\lah{n}{k}_{S,r}x^k.$$
Therefore, we obtain the  equality:
$$X = \La_{S,r}^{-1}\mathcal{L}_{S,r},$$
where $X=[1, x, x^2, \dots]^T$ and $\mathcal{L}_{S,r}=[L_{0,S,r}(x), L_{1,S,r}(x), L_{2,S,r}(x), \dots]^T$. Further,  $X=\F_{S,r}\mathcal{L}_{S,r}$ and
$$x^n=\sum_{k=0}^n \lah{n}{k}_{S,r}^{-1} L_{k,S,r}(x).$$
 Therefore,
 \begin{align}
 L_{n,S,r}(x)=x^n-\sum_{k=0}^{n-1}\lah{n}{k}_{S,r}^{-1} L_{k,S,r}(x), \quad n\geq 0.\label{iddet}
 \end{align}
 From the above identity we obtain a determinantal identity for the polynomials $L_{n,S,r}(x)$.
 \begin{theorem}\label{teo3}
 For all $S\subseteq \Z^+$ with $1\in S$, the $(S,r)$-Lah polynomials satisfy
 $$L_{n,S,r}(x)=(-1)^n\begin{vmatrix}
 1 & x & & \cdots && x^{n-1} & x^n\\
  1 & \lah{1}{0}_{S,r}^{-1} & &\cdots && \lah{n-1}{0}_{S,r}^{-1} & \lah{n}{0}_{S,r}^{-1}\\
    0 & 1 && \cdots && \lah{n-1}{1}_{S,r}^{-1}  & \lah{n}{1}_{S,r}^{-1} \\
   \vdots &  & & \cdots& & & \vdots\\
  0 & 0  && \cdots & & 1 & \lah{n}{n-1}_{S,r}^{-1}\\
 \end{vmatrix}.$$
 \end{theorem}
 \begin{proof}
This identity follows from Equation \eqref{iddet} and by expanding the determinant by the last column.
 \end{proof}

 For example, if $S=\{1, 2, 5\}$ and $r=2$, then
\begin{align*}
\La_{\{1, 2, 5 \}, 2}&=\left\langle \left(1 + 2x + 5x^4\right)^2 , x + x^2+x^5 \right\rangle\\
&=\left(\begin{array}{ccccccccc}
 1 & 0 & 0 & 0 & 0 & 0 & 0 & 0 & 0 \\
 4 & 1 & 0 & 0 & 0 & 0 & 0 & 0 & 0 \\
 8 & 10 & 1 & 0 & 0 & 0 & 0 & 0 & 0 \\
 0 & 48 & 18 & 1 & 0 & 0 & 0 & 0 & 0 \\
 240 & 96 & 156 & 28 & 1 & 0 & 0 & 0 & 0 \\
 2400 & 1320 & 720 & 380 & 40 & 1 & 0 & 0 & 0 \\
 0 & 24480 & 5760 & 3000 & 780 & 54 & 1 & 0 & 0 \\
 0 & 120960 & 126000 & 24360 & 9240 & 1428 & 70 & 1 & 0 \\
 1008000 & 0 & 1330560 & 483840 & 92400 & 23520 & 2408 & 88 & 1 \\
\end{array}
\right),
\end{align*}
and
\begin{multline*}
\F_{\{1, 2, 5\},2}=\left[\lah{n}{k}_{\{1, 2, 5\},2}^{-1}\right]_{n, k \geq 0}\\  \footnotesize
=\left(\begin{array}{ccccccccc}
 1 & 0 & 0 & 0 & 0 & 0 & 0 & 0 & 0 \\
 -4 & 1 & 0 & 0 & 0 & 0 & 0 & 0 & 0 \\
 32 & -10 & 1 & 0 & 0 & 0 & 0 & 0 & 0 \\
 -384 & 132 & -18 & 1 & 0 & 0 & 0 & 0 & 0 \\
 5904 & -2232 & 348 & -28 & 1 & 0 & 0 & 0 & 0 \\
 -110400 & 45000 & -7800 & 740 & -40 & 1 & 0 & 0 & 0 \\
 2422080 & -1051920 & 198000 & -21120 & 1380 & -54 & 1 & 0 & 0 \\
 -60641280 & 27921600 & -5624640 & 656040 & -48720 & 2352 & -70 & 1 & 0 \\
 1697351040 & -826801920 & 176863680 & -22176000 & 1812720 & -100464 & 3752 & -88 & 1
\end{array}\right).
\end{multline*}
The first few $(\{1, 2, 5\},2)$-Lah polynomials are
\begin{align*}
&1,\quad x+4, \quad x^2+10 x+8, \quad x^3+18 x^2+48 x, \quad x^4+28 x^3+156 x^2+96 x+240,\\
&x^5+40 x^4+380x^3+720 x^2+1320 x+2400, \quad x^6+54 x^5+780 x^4+3000 x^3+5760 x^2+24480 x, \\
&x^7+70 x^6+1428x^5+9240 x^4+24360 x^3+126000 x^2+120960 x, \ldots
\end{align*}
 In particular,
   \begin{align*}
   L_{6, \{1, 2, 5\}, 2}(x)&=-(x^5+40 x^4+380
   x^3+720 x^2+1320 x+2400) \\
   &=\begin{vmatrix}
 1 & x & x^2 & x^3 & x^4 & x^5 \\
 1 & -4 & 32 & -384 & 5904 & -110400 \\
 0 & 1 & -10 & 132 & -2232 & 45000 \\
 0 & 0 & 1 & -18 & 348 & -7800 \\
 0 & 0 & 0 & 1 & -28 & 740 \\
 0 & 0 & 0 & 0 & 1 & -40 \\
      \end{vmatrix}.
 \end{align*}

\textit{}\section{Combinatorial Interpretation: M\" obius inversion on Posets}

 In this section we provide another combinatorial interpretation for the Lah matrix $\L_{S,r}$, different from the combinatorial definition given at the beginning of Section \ref{rversion1}.  Using that interpretation, for the class of sets $S$ such that $S-1$ is an additive monoid, we construct a family of posets whose M\"obius function gives us the Lah inverse matrix $\F_{S,r}$.
 A fundamental role in our construction is played by the {\em asterisk lists}. An asterisk list of size $k$ is a list in $k$ symbols, plus an extra `ghost' element $\ast$. For example, $2\,4\,1\,\ast\,3\,5\,6$ is an asterisk list of length $6$. An asterisk list may be identified with an ordered pair of lists, $\tilde{\ell}=\ell_1\ast\ell_2=(\ell_1,\ell_2)$ (one, or even both of them, are allowed to be empty).
 \begin{proposition}\label{lahinterpretation}
 	The $(S,r)$-Lah number
 	$$\lah{n}{k}_{S,r}$$
 	counts the pairs of the form $(\tilde{\boldsymbol{\ell}},a)$, where
 	\begin{enumerate}
 		\item The first component is an $r$-tuple of asterisk lists, $\tilde{\boldsymbol{\ell}}=(\tilde{\ell_1},\tilde{\ell}_2,\dots,\tilde{\ell_r})$. The size of each component $\tilde{\ell_i}$ is in $S-1$. 
 		\item The second component of the pair, $a$, is a $k$ partition of non-empty lists, $a=\{\ell_1,\ell_2,\dots,\ell_k\}$. The size of each list in $a$ is in $S$. 
 		\item The sum of the sizes of the lists in the whole pair $(\tilde{\boldsymbol{\ell}},a)$ is  $n$.  
 		 
 	\end{enumerate}  
 \end{proposition}
\begin{proof}
	Let $a'=\{\ell'_1,\ell'_2,\dots,\ell'_r,\ell'_{r+1},\ell'_{r+2},\dots\ell'_{r+k}\}$ be a partition of lists as in the definition of $\lah{n}{k}_{S,r}$. The elements of $a'$ are ordered in such a way that for $i=1,2,\dots, r$ the list $\ell_i$ contains the element $i$. For $i=1,2,\dots,r$, define $\tilde{\ell}_i$ to be the asterisk list obtained by substituting the element $i$ in $\ell_i$ by the asterisk $\ast$. Then make $a:=\{\ell'_{r+1},\ell' _{r+2},\dots,\ell'_{r+k}\}$. The correspondence $a'\mapsto (\tilde{\boldsymbol{\ell}},a)$ is clearly a bijection. 
\end{proof}
The pair $(\tilde{\boldsymbol{\ell}},a)$ will be denoted by separating the $r$-tuple $\tilde{\boldsymbol{\ell}}$ from $a$ by a double bar, and the elements of $a$ by simple bars, $\tilde{\boldsymbol{\ell}}||a$.  

As an example of the notation, the pair $$((1\,3\,2\,\ast,\,\ast\, 5,\,4\,9\ast\,7\,6),\{8\, 11,12\, 10\})$$ will be written as
$$(1\,3\,2\,\ast,\,\ast\, 5\,,4\,9\ast\,7\,6)||8\, 11|12\, 10.$$

\begin{definition}\label{Sproperties}
 	A subset $S$ of $\mathbb{Z}^+$ such that $S-1$ is an additive monoid will be called a $^+ 1$ monoid.
 \end{definition}
For example,  the set $S$ of odd integers is a $^+1$ monoid, since $S-1$, the set of even integers is an additive monoid. More generally, for a positive integer $m$, the set of multiples of $m$ plus one is a $^+1$ monoid.
\begin{proposition}\label{Sprop}
	Let $S$ be a $^+1$ monoid. Then 
	\begin{enumerate}\item If $s_1, s_2,\dots, s_t$ and $t$ are all elements of $S$, then $s_1+s_2+\dots+s_t$ is in $S$.
		\item If $s_1, s_2,\dots, s_{t-1}$ are elements of $S$, and $t$ is also in $S$ (equivalently, $t-1\in S-1$), then $s_1+s_2+\dots+s_{t-1}$ is in $S-1$.
	\end{enumerate}
\end{proposition}
		\begin{proof}
		(1)	Since $S-1$ is a monoid, we have that
		$$(s_1-1)+(s_2-1)+\dots+(s_t-1)+(t-1)=s_1+s_2+\dots+s_t-1\in S-1,$$
		hence, $s_1+s_2+\dots+s_t\in S.$
		
		(2) Since $S-1$ is a monoid, $0\in S-1$, and hence, $1\in S$. By (1), making, $s_t=1$ we get
		$$s_1+s_2+\dots+s_t=s_1+s_2+\dots+s_{t-1}+1\in S\Rightarrow s_1+s_2+\dots+s_{t-1}\in S-1.$$
		\end{proof}
The $^+1$ monoids are of independent interests. Applications in the construction of posets defined on compositions and combinatorial interpretations of its M\" obius function will be given in a forthcoming paper.

 We let $\mathfrak{L}_{S,r}[n]$ denote the set of pairs $\tilde{\boldsymbol{\ell}}||a$ on the set of labels $[n]=\{1,2,\dots,n\}$, and by $\mathfrak{L}_{S,r}[n,k]$ the same kind of pairs such that $a$ has exactly $k$ lists. By Proposition \ref{lahinterpretation}, we have $$|\mathfrak{L}_{S,r}[n,k]|=\lah{n}{k}_{S,r}.$$ 
 
 Let $\ell_1$ and $\ell_2$ be two lists over disjoint sets. We denote by $\ell_1+\ell_2$ the concatenation of both lists (also denoted by juxtaposition $\ell_1\ell_2$).
 
 Our objective is to construct a partial order, $\leq$, on $\mathfrak{L}_{S,r}[n]$ such that
 \begin{enumerate}
 	\item The poset $(\mathfrak{L}_{S,r}[n],\leq )$ would have a zero $\widehat{0}=(\ast,\ast,\dots,\ast)||1|2|\dots|n$.
  \item Denoting by $\mu$ the M\"obius function of  $\mathfrak{L}_{S,r}[n]$, the M\"obius cardinal of $\mathfrak{L}_{S,r}[n,k]$ on the partial order:
 \begin{equation}\label{eq.mobiuscard}|\mathfrak{L}_{S,r}[n,k]|_{\mu}:=\sum_{\tilde{\boldsymbol{\ell}}||a\in \mathfrak{L}_{S,r}[n,k]}\mu(\hat{0},\tilde{\boldsymbol{\ell}}||a)\end{equation}
 would give us the entries of the inverse Lah matrix:
  \begin{equation}\label{eq.invmat}|\mathfrak{L}_{S,r}[n,k]|_{\mu}=\F_{S,r}[n,k].\end{equation}
 \end{enumerate}
For our purposes we need some definitions.
\begin{definition}(Asterisk product) For two asterisk lists $\tilde{\ell}=\ell_1\ast\ell_2$, $\tilde{\ell}'=\ell_1'\ast\ell'_2$, we define the product $\tilde{\ell}\circledast \tilde{\ell}'$ to be the asterisk list obtained by the substitution of the asterisk symbol in the first list by the second list
	$$\tilde{\ell}\circledast \tilde{\ell}':=\ell_1\ell_1'\ast\ell'_2\ell_2.$$
	For two $r$-tuples $\tilde{\boldsymbol{\ell}},\, \tilde{\boldsymbol{\ell}}'$ of asterisk lists we define the asterisk product to be the $r$-tuple of the asterisk products of the components:
	$$\tilde{\boldsymbol{\ell}}\circledast \tilde{\boldsymbol{\ell}}':=(\tilde{\ell}_1\circledast\tilde{\ell}'_1,\tilde{\ell}_2\circledast \tilde{\ell}'_2,\dots,\tilde{\ell}_r\circledast \tilde{\ell}'_r).$$
\end{definition}
Observe that, since $S-1$ is a monoid, the operation $\circledast$ is closed under $r$-tuples whose component sizes are all in $S-1$. It is easy to check that the product $\circledast$ satisfies the cancellation law,
\begin{equation}\label{leftcancellation}\tilde{\boldsymbol{\ell}}\circledast\tilde{\boldsymbol{\ell}}'=\tilde{\boldsymbol{\ell}}\circledast\tilde{\boldsymbol{\ell}}''\Rightarrow\tilde{\boldsymbol{\ell}}'=\tilde{\boldsymbol{\ell}}''.  
\end{equation}  
\begin{definition}
Let $a$ be a partition of lists such that $|a|\in S-1$, and $\tilde{\ell}$ an asterisk list. We say that \emph{ $\tilde{\ell}$ can be constructed from $a$} if it can be obtained by concatenation (in any order) of the lists in $a$ together with the asterisk symbol $\ast$. More generally, we say that an $r$-tuple $\tilde{\boldsymbol{\ell}}$ of asterisk lists can be constructed from $a$  if $a$ can be written as a disjoint union $a=\uplus_{i=1}^r a_i$ (some of them may be empty) such that each $ \tilde{\ell}_i$ can be constructed from $a_i$, for every $i=1,2,\dots,r$. 
\end{definition}
 In the previous definition, if all the sizes of the lists in $a$ are in $S$ and  $\tilde{\ell}$ can be constructed from $a$, then the size of $\tilde{\ell}$ is in $S-1$. This claim follows easily from Proposition \ref{Sprop} (2). Similarly, if $\tilde{\boldsymbol{\ell}}$ can be constructed from a, the size of every component of $\tilde{\boldsymbol{\ell}}$ is in $S-1$.

 We first define the partial order for the special case $r=0$. 
\begin{definition}\label{curlyorder}
	Let $a$ and $a'$ be two partitions of lists, $a$ in $\mathfrak{L}_{S,0}[n]$. We say that $a\preccurlyeq  a'$ if every list $\ell$ in $a'$ is the concatenation of $s_{\ell}$ lists of $a$, $s_{\ell}$ being an element of $S$.
\end{definition} 
\begin{remark}\label{curlyorderr}From Proposition \ref{Sprop} (1), we have that  all the sizes of the lists in $a'$ are also in $S$. Hence, $\preccurlyeq$ is a well defined relation on the set $\mathfrak{L}_{S,0}[n]$.
It is easy to check that $\preccurlyeq$  is an order relation with  the partition of singleton lists $1|2|\dots|n$ as the zero element. 
\end{remark}
Now, we are ready to define the partial order on $\mathfrak{L}_{S,r}[n]$ for general $r$.

\begin{definition}\label{theorder}
	Let $\tilde{\boldsymbol{\ell}}||a$ and $\tilde{\boldsymbol{\ell}}'||a'$ be two elements of $\mathfrak{L}_{S,r}[n]$. We say that  $\tilde{\boldsymbol{\ell}}||a\leq \tilde{\boldsymbol{\ell}}'||a'$ if there exists a subset $a''$ of $a$, and $\tilde{\boldsymbol{\ell}}''$ constructed from $a''$ such that 
	\begin{enumerate}
		\item $a-a''\preccurlyeq a'$.
		\item $\tilde{\boldsymbol{\ell}}'=\tilde{\boldsymbol{\ell}}\circledast \tilde{\boldsymbol{\ell}}''.$
			\end{enumerate}\end{definition}
The intuition behind this partial order is the following. We get up in the poset in two ways,
\begin{enumerate}
\item One is by concatenating lists in the righthand side.
\item The other is by moving elements from the right to the left. This is done by constructing first an $r$-tuple of asterisk lists from the elements to be moved, and then inserting it in the left hand side by the operation $\circledast$.\end{enumerate} The partial order is then obtained by  the iteration of the two ways of going up.

\begin{example}\label{ex.order}Let $S=\mathcal{O}$ and $r=2$.
Let $\ell_1,\ell_2,\dots,\ell_{12}$ be linear orders of odd size. By the definition of $\leq$, we have that
\begin{equation*}
	(\ell_{11}\ast \ell_{10},\ell_{12}\ast)||\ell_1|\ell_2|\dots|\ell_9\leq 		(\ell_{11}\ell_1\ast\ell_2 \ell_{10},\ell_{12}\ell_5\ast\ell_4)||\ell_3\ell_7\ell_6|\ell_8|\ell_9
	\end{equation*}
	\noindent because $(\ell_{11}\ell_1\ast\ell_2 \ell_{10},\ell_{12}\ell_5\ast\ell_4)=(\ell_{11}\ast \ell_{10},\ell_{12}\ast)\circledast (\ell_1\ast\ell_2,\ell_5\ast\ell_4)$, $(\ell_1\ast\ell_2,\ell_5\ast\ell_4)$ constructed from $a''= \ell_1|\ell_2|\ell_4|\ell_5$ and 
	$$a-a''=\ell_3|\ell_6|\ell_7|\ell_8|\ell_9\preccurlyeq \ell_3\ell_7\ell_6|\ell_8|\ell_9.$$ 
	 
\end{example}

It is not difficult to verify that the poset $\mathfrak{L}_{S,r}[n]$ has a zero $(\ast,\ast,\dots,\ast)||1|2|\dots|n$. 
\begin{example}
Let us consider the case $S=\mathcal{O}$, and $r=2$. The zero of the poset $\mathfrak{L}_{\mathcal{O},2}[3]$ is $(\ast,\ast)||1|2|3$. The asterisk lists in the left hand side are allowed only to have even size, while those in the right hand size only odd size (one or three). Then, there are $3!=6$ elements covering $\widehat{0}$ that we can get without moving elements from the right to the left, $(\ast,\ast)||\sigma_1\,\sigma_2\,\sigma_3$, all the  permutations of $1\,2\,3$. We can move only two elements from the left to the right, to one of the two components. The number of ways of choosing them is $\binom{3}{2}=3$. Assume we are moving $1$ and $2$. The  asterisk pairs of lists that we can construct from them are $3!+3!=12$. The maximal elements of this form are:
\begin{eqnarray*}(1 2\ast,\ast)||3, (1\ast 2,\ast)||3, (2 1\ast, \ast )||3, (2\ast 1,\ast)||3, (\ast 1 2,\ast)||3, (\ast 2 1,\ast)||3, &&\\  
 (\ast, 1 2\ast)||3, (\ast,1\ast 2)||3, (\ast,2 1\ast )||3, (\ast, 2\ast 1)||3, (\ast, \ast 1 2)||3, (\ast, \ast 2 1)||3.&&\end{eqnarray*}
We have $3\times 12=36$ of such kind of maximal elements. Then, the poset has
   $42$ maximal  elements, all of them covering $\widehat{0}$. We have that $|\mathfrak{L}_{\mathcal{O},2}[3,3]|_{\mu}=1$, $|\mathfrak{L}_{\mathcal{O},2}[3,2]|_{\mu}=0$ (since $\mathfrak{L}_{\mathcal{O},2}[3,2]=\emptyset$),  $|\mathfrak{L}_{\mathcal{O},2}[3,1]|_{\mu}=-42$, and $|\mathfrak{L}_{\mathcal{O},2}[3,0]|_{\mu}=0$, because $\mathfrak{L}_{\mathcal{O},2}[3,0]=\emptyset$. 
   
 Since $\mathfrak{L}_{\mathcal{O},2}[4,1]=\emptyset=\mathfrak{L}_{\mathcal{O},2}[4,3]$ we have $|\mathfrak{L}_{\mathcal{O},2}[4,1]|_{\mu}=|\mathfrak{L}_{\mathcal{O},2}[4,3]|_{\mu}=0$. The M\"obius function of the intervals of the form $[\widehat{0},(\tilde{\ell},*)||\emptyset\,]$ and $[\widehat{0},(\ast,\tilde{\ell})||\emptyset\,]$, $|\tilde{\ell}|=4$,  is $2$ (see Fig. \ref{fig.pset1} (a)). There are $5!+5!=240$  of them. The M\"obius function of the intervals of the form $[\widehat{0},(\tilde{\ell}_1,\tilde{\ell}_2)||\emptyset\,]$ (See Fig. \ref{fig.pset1} (b)), is equal to $1$, and there are  $6\times 6\times 6=216$ of them. Hence, $|\mathfrak{L}_{\mathcal{O},2}[4,0]|_{\mu}=240\times 2+216\times 1=480+216=696$. The M\"obius cardinal $|\mathfrak{L}_{\mathcal{O},2}[4,2]|_{\mu}$ is easier to compute and left to the reader.
\begin{figure}[]\begin{center}
	\includegraphics[width=0.9 \linewidth]{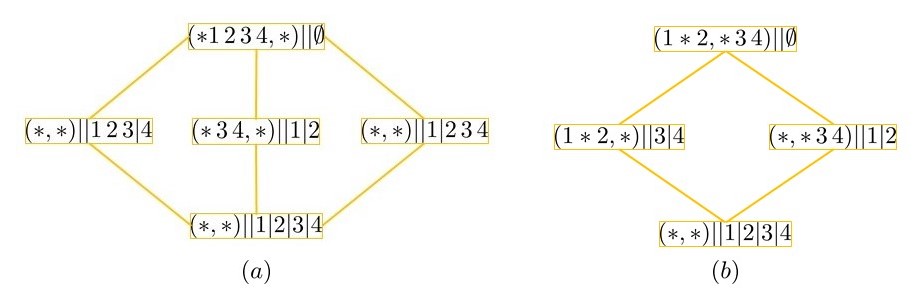}
\end{center}\caption{Intervals  $[\widehat{0},(\ast\,1\,2\,3\,4,\ast)||\emptyset\,]$ and $[\widehat{0},(1\,\ast\,2,\ast\,3\,4)||\emptyset]$. }\label{fig.pset1}
\end{figure}
  
See the first few rows of the matrix $\F_{\mathcal{O},2}$. The fourth and fifth row agree with Eq. (\ref{eq.invmat}) and our computations. 
$$\F_{\mathcal{O},2}=\left[\lah{n}{k}_{\mathcal{O},2}^{-1}\right]_{n, k \geq 0}= \left(
\begin{array}{cccccccc}
 1 & 0 & 0 & 0 & 0 & 0 & 0 & 0 \\
 0 & 1 & 0 & 0 & 0 & 0 & 0 & 0 \\
 -12 & 0 & 1 & 0 & 0 & 0 & 0 & 0 \\
 0 & -42 & 0 & 1 & 0 & 0 & 0 & 0 \\
 696 & 0 & -96 & 0 & 1 & 0 & 0 & 0 \\
 0 & 4440 & 0 & -180 & 0 & 1 & 0 & 0 \\
 -93600 & 0 & 16560 & 0 & -300 & 0 & 1 & 0 \\
 0 & -887040 & 0 & 47040 & 0 & -462 & 0 & 1 \\
\end{array}
\right).$$   
\end{example}
If $n\in S-1$, all the maximal elements of $\mathfrak{L}_{S,r}[n]$ are of the form $\tilde{\boldsymbol{\ell}}||\emptyset$.
 Another consequence of the Definition \ref{theorder} is that $a\preccurlyeq a'$ is equivalent to say that $(\ast,\ast,\dots,\ast)||a \leq (\ast,\ast,\dots,\ast)||a'$. Hence, given $S$, the partial order$\preccurlyeq$ on $\mathfrak{L}_{S,0}[n]$, is the same partial order as $\leq$ restricted to pairs where the left hand side is the trivial tuple $(\ast,\ast,\dots,\ast)$. 
 \begin{proposition}\label{prop:coideal}Let $\tilde{\boldsymbol{\ell}}||a$ be a fixed element of $\mathfrak{L}_{S,r}[n,k]$, and $j$ a non-negative integer, $k\geq j$. The set $\mathcal{C}_{\tilde{\boldsymbol{\ell}}||a}[j]=\{\tilde{\boldsymbol{\ell}}'||a':\tilde{\boldsymbol{\ell}}'||a'\geq \tilde{\boldsymbol{\ell}||a,\; |a'|=j\}}$ is equipotent with  $\mathfrak{L}_{S,r}[k,j],$
 	\begin{equation*}
 	|\mathcal{C}_{\tilde{\boldsymbol{\ell}}||a}[j]|=|\mathfrak{L}_{S,r}[k,j]|.  
 	\end{equation*}
 \end{proposition} 
 \begin{proof}
	Choose one element $\tilde{\boldsymbol{\ell}}'||a'$ in $\mathcal{C}_{\tilde{\boldsymbol{\ell}}||a}[j]$. By the left cancellation law (Eq. \ref{leftcancellation}) and Definition \ref{theorder} there exists a unique $\tilde{\boldsymbol{\ell}}''$ such that $\tilde{\boldsymbol{\ell}}'=\tilde{\boldsymbol{\ell}}\circledast\tilde{\boldsymbol{\ell}}'',$ $\tilde{\boldsymbol{\ell}}''$ constructed from some subset $a''$ of $a$ and such that $a-a''\preccurlyeq a'.$ Hence, the correspondence $\tilde{\boldsymbol{\ell}}'||a'\mapsto \tilde{\boldsymbol{\ell}}''||a'$ is a bijection.  Ordering the elements of $a$, $a=\{\ell_1,\ell_2,\dots,\ell_k\}$, we substitute by $i$ any appearance of $\ell_i$ as a segment either in the components of $\tilde{\boldsymbol{\ell}}''$ or in any of the orders in $a'$. In that way we obtain a pair $\tilde{\boldsymbol{\ell}}'''||a'''$ in $\mathfrak{L}_{S,r}[k,j]$. The correspondence obtained by the composition
	$$\tilde{\boldsymbol{\ell}}'||a'\mapsto \tilde{\boldsymbol{\ell}}''||a'\mapsto  \tilde{\boldsymbol{\ell}}'''||a'''$$ is a bijection. We can go back by restituting $\ell_i$ in the place of $i$ and then making $\tilde{\boldsymbol{\ell}}'=\tilde{\boldsymbol{\ell}}\circledast \tilde{\boldsymbol{\ell}}''.$
In Example \ref{ex.order}, $\tilde{\boldsymbol{\ell}}'||a'=(\ell_{11}\ell_1\ast\ell_2 \ell_{10},\ell_{12}\ell_5\ast\ell_4)||\ell_3\ell_7\ell_6|\ell_8|\ell_9$ is sent by the above bijection to $(1\,\ast\,2, 5\,\ast\,4)||3\,7\,6|8|9$ in $\mathfrak{L}_{\mathcal{O},2}[9,3]$, since $\tilde{\boldsymbol{\ell}}||a=(\ell_{11}\ast \ell_{10},\ell_{12}\ast)||\ell_1|\ell_2|\dots|\ell_9\in\mathfrak{L}_{\mathcal{O},2}[n,9]$ for some undetermined $n$.\end{proof}

\begin{theorem}
	The M\" obius function of the posets $\mathfrak{L}_{S,r}[n]$, $n\in \mathbb{Z}^+$,  gives us the entries of the inverse of Lah matix,
	\begin{equation*}
	|\mathfrak{L}_{S,r}[n,k]|_{\mu}=\F_{S,r}[n,k],
	\end{equation*}
where $|\mathfrak{L}_{S,r}[n,k]|_{\mu}$ is as in Eq. (\ref{eq.mobiuscard}).   
\end{theorem}
\begin{proof}It is enough to prove that for every $0\leq j\leq n$, $$\sum_{j\leq k\leq n}|\mathfrak{L}_{S,r}[n,k]|_{\mu}|\mathfrak{L}_{S,r}[k,j]|=\delta_{n,j}.$$
	Let $\tilde{\boldsymbol{\ell}}'||a'$ be an element of $\mathfrak{L}_{S,r}[k,j]$. By the properties of the M\"obius function we have that
		\begin{equation*}
		\sum_{\widehat{0}\leq\boldsymbol{\tilde{\ell}}||a\leq \boldsymbol{\tilde{\ell}}'||a'}\mu(\widehat{0},\boldsymbol{\tilde{\ell}}||a) =\delta_{n,j}.
		\end{equation*}  
	 Summing over all the elements of $\mathfrak{L}_{S,r}[k,j]$, interchanging sums, and classifying by the size of $a$, we get
	 \begin{eqnarray*}
	 \delta_{n,j}&=&\sum_{\boldsymbol{\tilde{\ell}}'||a'\in \mathfrak{L}_{S,r}[n,j]}\;\sum_{\widehat{0}\leq\boldsymbol{\tilde{\ell}}||a\leq \boldsymbol{\tilde{\ell}}'||a'}\mu(\widehat{0},\boldsymbol{\tilde{\ell}}||a)\\&=&\sum_{\widehat{0}\leq \boldsymbol{\tilde{\ell}}||a}\;\sum_{\boldsymbol{\tilde{\ell}}'||a'\geq \boldsymbol{\tilde{\ell}}||a,\,|a'|=j}\mu(\widehat{0},\boldsymbol{\tilde{\ell}}||a)\\&=&\sum_{j\leq k\leq n}\;\sum_{\widehat{0}\leq \boldsymbol{\tilde{\ell}}||a,\, |a|=k}\mu(\widehat{0},\boldsymbol{\tilde{\ell}}||a)|\{\boldsymbol{\tilde{\ell}}'||a':\boldsymbol{\tilde{\ell}}'||a'\geq \boldsymbol{\tilde{\ell}}||a,\,|a'|=j\}|\\
	  &=&\sum_{j\leq j\leq n}\;|\mathfrak{L}_{S,r}[n,k]|_{\mu}||\mathfrak{L}_{S,r}[k,j]|.
	 \end{eqnarray*}
 The last identity is obtained from Proposition \ref{prop:coideal}.
	\end{proof}

\section{Ordered $(S,r)$-Partitions}

The Fubini numbers $F_{n}$  count the number of ordered set partitions.  It is natural to generalize them by  restricting the size of the blocks used in the partitions by a given set $S$, with $r$ special elements. This gives the  \emph{$(S,r)$- Fubini numbers}, $F_{n, S, r}$, where
the size of each block is contained in the set $S\subseteq \Z^+$ and the first $r$ elements are contained in distinct blocks.

From the above definition  it is clear that
\begin{equation}\label{SrFubini}
F_{n, S,r}=\sum_{k=0}^n (k+r)! { n \brace k}_{S,r},
\end{equation}
where ${ n \brace k}_{S,r}$ are the $(S,r)$-Stirling numbers of the second kind. This sequence was recently studied in \cite{BenyiRamirez2}.  The sequence ${ n \brace k}_{S,r}$ counts the total number of set partitions of $n+r$ elements  into $k+r$ non-empty blocks such that the cardinality of each block is contained in the set $S$ and the first $r$ elements are  in distinct blocks.

For the $(S,r)$-Stirling numbers of the second kind  we have the construction
$$\mbox{SET}_k(\mbox{SET}_S(\mathcal{X})) \times \mbox{SEQ}_r(\mbox{SET}_{S-1}(\mathcal{X})).$$
Then from the symbolic method we obtain the exponential generating function
\begin{equation}\label{Sgenfunc2}
\sum_{n=k}^\infty { n \brace k }_{S,r} \frac{x^n}{n!}=\frac{1}{k!}\left(\sum_{s \in S} \frac{x^{s}}{s!}\right)^k\left(\sum_{s\in S} \frac{x^{s-1}}{(s-1)!}\right)^r.
\end{equation}

 Let $g_{S,r,t}(x)$ be the function defined by
$$g_{S,r,t}(x):=(1+E_{S}(x))^t\left(\sum_{s\in S} \frac{x^{s-1}}{(s-1)!}\right)^r,$$
where $E_S(x)=\sum_{s\in S} \frac{x^{s}}{s!}$.
Then from Theorem 8 of \cite{mihoubi-2017a} we have
 \begin{align*}
\left. \frac{d^n}{dx^n}g_{S,r,t}(x)\right |_{x=0}:=g_{S,r,t}^{(n)}(0)=\sum_{k=0}^n{n \brace k}_{S, r}(t)_k.
\end{align*}

\begin{theorem}\label{gfunfubini}
The exponential generating function for the $(S,r)$-restricted Fubini numbers is
\begin{align*}
\sum_{n=0}^\infty F_{n, S, r}\frac{x^n}{n!}=\frac{r!}{(1-E_S(x))^{r+1}}\left({\sum_{s\in S}\frac{x^{s-1}}{(s-1)!}}\right)^r,
\end{align*}
where
$$
E_S(x)=\sum_{s\in S} \frac{x^{s}}{s!}.
$$
\end{theorem}
\begin{proof}
From  Equations  \eqref{SrFubini} and \eqref{Sgenfunc2} we have
\begin{align*}
\sum_{n=0}^\infty F_{n, S, r}\frac{x^n}{n!}&=\sum_{n=0}^\infty \sum_{k=0}^n (k+r)!{ n \brace k}_{S,r} \frac{x^n}{n!}=\sum_{k=0}^\infty (k+r)!\sum_{n=k}^\infty  { n \brace k}_{S,r} \frac{x^n}{n!}\\
&=\frac{r!}{(1-E_S(x))^{r+1}}\left(\sum_{s\in S} \frac{x^{s-1}}{(s-1)!}\right)^r.
\end{align*}
\end{proof}
We can also derive the generating function using the symbolic method.
We obtain an ordered partition of $(n+r)$ elements into $(k+r)$ blocks such that the first $r$ elements are in distinct blocks the following way: we take a sequence of ordinary blocks, eventually empty, then a special block, again a sequence of ordinary blocks, eventually empty, followed again by a special block and so on. There are $r$ special blocks, among the $r+1$ sequences of ordinary blocks. Finally, we put one of the $r$ elements into each special block, which can be done in $r!$ ways.
This leads to the construction:
\[\mbox{SEQ}(\mbox{SET}_S(\mathcal{X}))\times\mbox{SET}_{S-1}(\mathcal{X}) \mbox{SEQ}(\mbox{SET}_S(\mathcal{X}))\times\cdots\times \mbox{SET}_{S-1}(\mathcal{X})\times \mbox{SEQ}(\mbox{SET}_S(\mathcal{X})),\]
which translates by the symbolic method into
\[\frac{1}{(1-E_S(x))^{r+1}}\left(\sum_{s\in S} \frac{x^{s-1}}{(s-1)!}\right)^r.\]
Multiplying with $r!$ we obtain the generating function.

In Theorem \ref{recSrFubini} we give a recurrence relation for the sequence $F_{n, S,r}$
\begin{theorem}\label{recSrFubini}
Let $n\in \N$. Then the $(S,r)$-Fubini numbers satisfy the recurrence relation
$$F_{n,S,r}=\sum_{s\in S} \binom ns F_{n-s,S,r}+r\sum_{s\in S} \binom{n}{s-1}F_{n-(s-1),S,r-1}.$$
\end{theorem}
\begin{proof}
The left-hand side counts the $(S,r)$-ordered set partitions of $[n+r]$. Consider the last block of an ordered set partition. Assume the last block  is non-special and has $s$ elements, for some $s\in S$.  This is done in $\binom ns F_{n-s,S,r}$ ways. If the last block is special then there are $r\binom{n}{s-1} F_{n-(s-1),S,r}$ options. Summing over $s$ gives the identity.
\end{proof}

\begin{theorem}\label{rSfubiniid1}
The $(S,r)$-restricted Fubini numbers satisfy
\begin{align}
F_{n, S, r}=\frac{r!}{2^{r+1}}\sum_{\ell=0}^\infty \frac{1}{2^\ell}\binom{r+\ell}{\ell}\sum_{k=0}^n {n \brace k}_{S,r}(\ell)_k.
\end{align}
\end{theorem}
\begin{proof}
From Theorem \ref{gfunfubini} we have
\begin{align*}
\sum_{n=0}^\infty F_{n, S, r}\frac{x^n}{n!}=&\frac{r!}{(2-(E_S(x)+1))^{r+1}}\left({\sum_{s\in S}\frac{x^{s-1}}{(s-1)!}}\right)^r\\
=&\frac{r!}{2^{r+1}}\frac{1}{\left(1-\left(\frac{E_S(x)+1}{2}\right)\right)^{r+1}}\left({\sum_{s\in S}\frac{x^{s-1}}{(s-1)!}}\right)^r\\
=&\frac{r!}{2^{r+1}}\sum_{\ell=0}^\infty\binom{r+\ell}{\ell}\left(\frac{E_S(x)+1}{2}\right)^{\ell}\left({\sum_{s\in S}\frac{x^{s-1}}{(s-1)!}}\right)^r\\
=&\frac{r!}{2^{r+1}}\sum_{\ell=0}^\infty\frac{1}{2^\ell}\binom{r+\ell}{\ell}g_{\ell,S,r}(x).
\end{align*}
Since
$$[x^n]g_{S,r,k}(x)=\frac{1}{n!}\left.\frac{d^n}{dx^n}g_{S,r,k}(x)\right |_{x=0}=\frac{1}{n!}\sum_{\ell=0}^n{n \brace \ell}_{S,r} (k)_\ell,
$$

we have
\begin{align*}
\sum_{n=0}^\infty F_{n, S, r}\frac{x^n}{n!}=&\frac{r!}{2^{r+1}}\sum_{\ell=0}^\infty\frac{1}{2^\ell}\binom{r+\ell}{\ell} \sum_{n=0}^\infty \left(\frac{1}{n!}\sum_{k=0}^n { n \brace k }_{S,r}(\ell)_k\right)x^n.
\end{align*}
Comparing the $n$-th coefficient we obtain the desired result.
\end{proof}

Note that the above equality is a generalization of the identity (cf.\ \cite[pp. 228]{comtet-1974a})
 $$F_n=\frac{1}{2}\sum_{k=0}^\infty \frac{k^n}{2^k}.$$

\section{Doubly Ordered $(S,r)$-Partitions}
In this section we consider ordered lists such that the order of the elements in each list matters.  This kind of ordered partitions are called \emph{doubly ordered partition} (cf.\ \cite{ManSha2}). In particular, we denote by $\D_{n, S,r}$ the total number of doubly ordered partitions of $[n+r]$ such that the size $s$ of each list belongs to a given set $S$ of positive integers and the first $r$ elements are in distinct blocks.  It is clear that
$$\D_{n,S,r}=\sum_{k=0}k!\lah{n}{k}_{S,r}.$$

\begin{theorem}\label{gfundoubly}
The exponential generating function for the sequence  $\D_{n, S,r}$ is
\begin{align*} 
\sum_{n=0}^\infty \D_{n, S, r}\frac{x^n}{n!}=\frac{r!}{(1-\sum_{s\in S}x^s)^{r+1}}\cdot\left(\sum_{s\in S}sx^{s-1}\right)^r.
\end{align*}
\end{theorem}
\begin{proof}
We obtain a doubly ordered partition of $[n+r]$ taking a sequence of ordinary lists, eventually empty, then a special list, again a sequence of ordinary lists, eventually empty, followed again by a special list and so on. There are $r$ special lists among the $r+1$ sequences of ordinary lists. We have to point to a gap in each special list where we can insert one of the special elements.
This leads to the construction:
\[\mbox{SEQ}(\mbox{SEQ}_S(\mathcal{X}))\times\Theta^*(\mbox{SEQ}_{S-1}(\mathcal{X})) \mbox{SEQ}(\mbox{SEQ}_S(\mathcal{X}))\times\cdots\times \Theta^*(\mbox{SEQ}_{S-1}\mathcal{X})\times \mbox{SEQ}(\mbox{SEQ}_S(\mathcal{X})),\]
which translates by the symbolic method into
\[
\frac{1}{(1-\sum_{s\in S}x^s)^{r+1}}\cdot\left(\sum_{s\in S}sx^{s-1}\right)^r.
\]
Multiplying with $r!$ we obtain the desired result.
\end{proof}

The proofs of the following two theorems  are analogous to the one given for Theorems \ref{recSrFubini}  and \ref{rSfubiniid1}.

\begin{theorem}
Let $n\in \N$. Then the sequence  $\D_{n, S,r}$  satisfies the recurrence relation
$$\D_{n,S,r}=\sum_{s\in S} (n)_s \D_{n-s,S,r} + r\sum_{s\in S} s(n)_{s-1}\D_{n-(s-1),S,r-1}.$$
\end{theorem}

\begin{theorem}
The sequence  $\D_{n, S,r}$ satisfies
\begin{align*}
\D_{n, S, r}=\frac{r!}{2^{r+1}}\sum_{\ell=0}^\infty \frac{1}{2^\ell}\binom{r+\ell}{\ell}\sum_{k=0}^n \lah{n}{k}_{S,r}(\ell)_k.
\end{align*}
\end{theorem}

\section{Acknowledgements}
The authors would like to thank the anonymous referees for carefully reading the paper and giving helpful comments and suggestions.
The third author was partially supported by Universidad Nacional de Colombia, Project No.  46240.


\begin{thebibliography}{10}

\bibitem{Barry} P.~Barry, On a family of generalized Pascal triangles defined by exponential Riordan arrays, \textit{J. Integer Seq.} \textbf{10} Article 07.3.5 (2007), 1--21.


\bibitem{Belbachir-2013}
H.~Belbachir and A.~Belkhir, Cross recurrence relations for $r$-Lah numbers, \textit{Ars Combin.} \textbf{110} (2013), 199--203.


\bibitem{Belbachir-2014} H.~Belbachir and I.~E.~Bousbaa, Combinatorial identities for the $r$-Lah numbers, \textit{Ars Combin.}  \textbf{115} (2014), 453--458.

\bibitem{Belbachir-2016}
H. Belbachir and  I.~E.~Bousbaa, Associated Lah numbers and $r$-Stirling numbers,
\newblock arXiv:1404.5573 (2014).

\bibitem{BenyiRamirez} B.~B\'enyi and  J.~L.~ Ram\'irez, Some applications of $S$-restricted set partitions, \textit{Period. Math. Hungar.}  \textbf{78} (2019), 110--127.

\bibitem{BenyiRamirez2} B.~B\'enyi, M. M\'endez,  J.~L.~ Ramirez, and T.~Wakhare, Restricted $r$-Stirling numbers and their combinatorial applications, \textit{Appl. Math. Comput.} \textbf{348} (2019), 186--205.

\bibitem{Broder} A.~Z.~Broder, The $r$-Stirling numbers, \textit{Discrete Math.} \textbf{49} (1984), 241--259.

\bibitem{Callan}
D.~Callan, Sets, lists and noncrossing partitions,  \textit{J. Integer Seq.} \textbf{11} Article 08.1.3 (2008), 1--7.


\bibitem{comtet-1974a}
L.~Comtet,
\newblock {\em Advanced Combinatorics}.
\newblock D. {R}eidel {P}ublishing {C}o. ({D}ordrecht, {H}olland), 1974.


\bibitem{Engbers-2016}
J.~Engbers, D.~Galvin, and C.~Smyth, Restricted Stirling and Lah numbers and their inverses, \textit{J. Combin. Theory Ser. A} \textbf{161} (2019), 271--298.


\bibitem{Flajolet}
P.~Flajolet and R.~Sedgewick,
\newblock {\em Analytic Combinatorics}, Cambridge University Press, Cambridge, 2009.


\bibitem{Heubach-2004}
S.~Heubach and T.~Mansour,
\newblock Compositions of $n$ with parts in a set. Congr. Numer. \textbf{168} (2004), 127--143.

\bibitem{ManSha}
T.~Mansour and M.~Shattuck, A generalized class of restricted Stirling and Lah numbers, \textit{Math. Slovaca} \textbf{68} (4) (2018), 727--740.

\bibitem{ManSha2}
T.~Mansour and M.~Shattuck,  A polynomial generalization of some associated sequences related to set partitions, \textit{Period. Math. Hungar.} \textbf{75} (2017), 398--412.


\bibitem{mihoubi-2017a}
M.~Mihoubi and M.~Rahmani, The partial $r$-{B}ell polynomials, \textit{Afr.~Math.} \textbf{28} (2017), 1167--1183.

\bibitem{Nyul-2015}
G.~Nyul and G.~R\'{a}cz, The $r$-Lah numbers, \textit{Discrete Math.} \textbf{338} (2015), 1660--1666.

\bibitem{Riordan}
J.~Riordan, \textit{Introduction to combinatorial analysis},  Dover Publications, 2002.

\bibitem{Riordan2} L.~W.~Shapiro, S.~Getu, W.~Woan, and L.~Woodson, The Riordan group, \textit{Discrete Appl. Math.} \textbf{34} (1991), 229--239.

\bibitem{Shattuck} M.~Shattuck, Some formulas for the restricted $r$-Lah numbers, \textit{Ann. Math. Inform.} \textbf{49} (2018), 123--140.

\bibitem{MS}
\newblock  M.~Shattuck and  C.~G.~Wagner, Parity theorems for statistics on lattice paths and Laguerre configurations, \textit{J. Integer Seq.} \textbf{8}  Article 05.5.1 (2005), 1--13.



\end{thebibliography}
\end{document}